\documentclass[preprint,12pt]{elsarticle}

\usepackage{color}
\usepackage{amssymb}
\usepackage{amsthm}
\usepackage{amsmath}
\usepackage{epic}
\usepackage{CJK}
\usepackage{setspace}
\usepackage{bm}
\newtheorem{thm}{Theorem}[section]

\newtheorem{lem}[thm]{Lemma}

\newtheorem{con}{Conjecture}

\journal{~~}

\begin{document}
\begin{spacing}{1.15}
\begin{CJK*}{GBK}{song}
\begin{frontmatter}
\title{\textbf{Bollob\'{a}s-Nikiforov conjecture holds asymptotically almost surely}}

\author[label a]{Chunmeng Liu\corref{cor}}\ead{liuchunmeng214@nenu.edu.cn/liuchunmeng0214@126.com}
\author[label b]{Changjiang Bu}\ead{buchangjiang@hrbeu.edu.cn}
\cortext[cor]{Corresponding author}

\address{
\address[label a]{Academy for Advanced Interdisciplinary Studies, Northeast Normal University, Changchun 130024, PR China}
\address[label b]{College of Mathematical Sciences, Harbin Engineering University, Harbin 150001, PR China}
}

\begin{abstract}
Bollob\'{a}s and Nikiforov (J. Combin. Theory Ser. B. 97 (2007) 859-865) conjectured that for a graph $G$ with $e(G)$ edges and the clique number $\omega(G)$, then
\begin{align*}
\lambda_{1}^{2}+\lambda_{2}^{2}\leq 2e(G)\left(1-\frac{1}{\omega(G)}\right), 
\end{align*}
where $\lambda_{1}$ and $\lambda_{2}$ are the largest and  the second largest eigenvalues of the adjacency matrix of $G$, respectively. 
In this paper, we prove that for a sequence of random graphs the conjecture holds true with probability tending to one as the number of vertices tends to infinity.
\end{abstract}

\end{frontmatter}

\section{Introduction}

The graphs under consideration in this paper are simple and undirected. 
For a graph $G$ with $n$ vertices, we  denote the number of edges in $G$ by $e(G)$. 
The clique number of $G$ is denote by $\omega(G)$.
Let $A(G)$ denote the adjacency matrix of $G$. 
The eigenvalues of $A(G)$ are denote by $\lambda_{1}\geq \lambda_{2}\geq\cdots\geq\lambda_{n}$. 

In 2007, Bollob\'{a}s and Nikiforov \cite{Nikiforov_2007} proposed the following conjecture.
\begin{con}\textup{\cite{Nikiforov_2007}}\label{con_1}
Let $G$ be a graph with $e(G)$ edges and the clique number $\omega(G)$. 
Then
\begin{align*}
\lambda_{1}^{2}+\lambda_{2}^{2}\leq 2e(G)\left(1-\frac{1}{\omega(G)}\right).
\end{align*}
\end{con}
This conjecture has been proven for several special cases, including triangle-free graphs \cite{Lin_2011}, regular graphs \cite{Zhang_2024}, weakly perfect graphs, and Kneser graphs\cite{Elphick_2024}.

In this paper, we obtain that the Bollob\'{a}s-Nikiforov conjecture holds true with probability tending to one as the number of vertices tends to infinity for a sequence of random graphs. 

Let $0<p<1$ be  a fixed probability and let $G_{n}$ denote a random graph with $n$ vertices where each edge appears with $p$, independently of all other edges. 
Specifically, the random variables $e_{ij}$ for $1\leq i<j$ are defined as follows:
\begin{align*}
e_{ij}=
\begin{cases}
1, \ \textup{if $(i,j)$ is an edge of $G_{n}$},\\
0, \ \textup{if $(i, j)$ is not an edge of $G_{n}$},
\end{cases}
\end{align*}
are independent random variables with $\mathbb{P}(e_{ij}=1)=p$, $\mathbb{P}(e_{ij}=0)=1-p$.
Let $G_{n}$ be the above random graph and the adjacency matrix given by $A(G_{n})=(e_{ij})$. 

\begin{thm}\label{thm_1}
Suppose that $G_{n}$ is a random graph as described above.
For $0<\epsilon<1$, there exists an $n^{\prime}$ such that for all $n>n^{\prime}$, 
\begin{align*}
\mathbb{P}\left(\lambda_{1}^{2}+\lambda_{2}^{2}\leq 2e(G_{n})\left(1-\frac{1}{\omega(G_{n})}\right)\right)\geq 1-\exp\left(-Cn(n-1)\right),
\end{align*}
where $C=\epsilon^{2}p^{2}$.
\end{thm}

The remainder of this paper is organized as follows: 
In Section 2, we introduce definitions and lemmas necessary for the proofs.
In Section 3, we present the proofs of the conclusions of this paper.

\section{Preliminaries}

Let $A_{n}=(\xi_{ij})$ be an $n\times n$ symmetric matrix where $\xi_{ii}=0$, and $\xi_{ij}$ for $i>j$ are independent random variables. 
Suppose that $\mathbb{P}(\xi_{ij}=1)=p$, and $\mathbb{P}(\xi_{ij}=0)=1-p$. 
For the random matrix $A_{n}$, Juh\'{a}sz \cite{Juhasz_1981} gave the following conclusion.
\begin{lem}\textup{\cite{Juhasz_1981}}\label{lem_Juhasz}
If $\lambda_{1}$ is the largest eigenvalue of the matrix $A_{n}$, then
\begin{align*}
\mathbb{P}\left(\lim_{n\rightarrow\infty}\frac{\lambda_{1}}{n}=p\right)=1.
\end{align*}
\end{lem}

For the second largest eigenvalue of $A_{n}$, a conclusion can be derived from F\"{u}redi and Koml\'{o}s \cite{Furedi_1981}.
\begin{lem}\textup{\cite{Furedi_1981}}\label{lem_Furefi}
For the second largest eigenvalue $\lambda_{2}$ of $A_{n}$, there exists a constant $C_{0}>0$ such that with probability tending to $1$,
\begin{align*}
\lambda_{2}\leq 2(p(1-p)n)^{1/2}+C_{0} n^{1/3}\log n.
\end{align*}
\end{lem}

\begin{lem}\textup{\cite{Grimmett_1975}}\label{lem_grimmett}
Let the random variable $\omega(G_{n})$ be the clique number of a random graph $G_{n}$ where each edge occurs with probability $p$. 
Then 
\begin{align*}
\mathbb{P}\left(\lim_{n\rightarrow\infty}\frac{\omega(G_{n})}{\log n}=\frac{2}{\log(1/p)}\right)=1.
\end{align*}
\end{lem}

\section{Proof of the theorem}

The main idea of the proof is based on \cite{Rocha_2020}.

\begin{lem}\label{lem_1}
Let $0<\epsilon<1$. 
If $p\in\left(0,\frac{(1-\epsilon)^{2}}{1+2\epsilon}\right]$, then there exists $n_{0}$ such that for all $n>n_{0}$, 
\begin{align*}
&(1+\epsilon)p^{2}n^{2}+4p(1-p)n+4C_{0}(p(1-p))^{1/2}n^{5/6}\log n+C_{0}^{2}n^{2/3}(\log n)^{2}\\
&\leq p(1-\epsilon)n(n-1)\left(1-\frac{\log(1/p)}{2(1-\epsilon)\log n}\right),
\end{align*}
where $C_{0}>0$ is a constant.
\end{lem}
\begin{proof}
Consider the expression
\begin{align*}
(1+\epsilon)p^{2}+4p(1-p)n^{-1}+4C_{0}(p(1-p))^{1/2}n^{-7/6}\log n+C_{0}^{2}n^{-4/3}(\log n)^{2}.
\end{align*}
To analyze this expression, we define a series of thresholds. 
Let $m_{0}=\frac{12(1-p)}{\epsilon p}$. 
For $n>m_{0}$, we observe that
\begin{align*}
4p(1-p)n^{-1}\leq \frac{1}{3}\epsilon p^{2}.
\end{align*}
We define $m_{1}=\left(\frac{12C_{0}(p(1-p))^{1/2}}{\epsilon p^{2}}\right)^{6}$. 
For $n>m_{1}$, it follows that
\begin{align*}
4C_{0}(p(1-p))^{1/2}n^{-7/6}\log n < 4C_{0}(p(1-p))^{1/2}n^{-1/6} \leq \frac{1}{3}\epsilon p^{2}.
\end{align*}
We set $m_{2}=\left(\frac{3C_{0}^{2}}{\epsilon p^{2}}\right)^{3}$. 
For $n>m_{2}$, we have
\begin{align*}
C_{0}^{2}n^{-4/3}(\log n)^{2} < C_{0}^{2}n^{-1/3} \leq \frac{1}{3}\epsilon p^{2},
\end{align*}
the first inequality follows from the fact that $(\log n)^{2}<n$ for $n\geq 1$.
By combining these observations, we can establish a critical threshold $n_{0}^{\prime}$ defined as the maximum of $m_{0}$, $m_{1}$ and $m_{2}$:
\begin{align*}
n_{0}^{\prime}=\max\left\{m_{0},m_{1},m_{2}\right\}
\end{align*}
For $n>n_{0}^{\prime}$, we obtain
\begin{align*}
(1+\epsilon)p^{2}+4p(1-p)n^{-1}+4C_{0}(p(1-p))^{1/2}n^{-7/6}\log n+C_{0}^{2}n^{-4/3}(\log n)^{2}\leq (1+\epsilon)p^{2}+\epsilon p^{2}.
\end{align*}
Given that $p\in\left(0,\frac{(1-\epsilon)^{2}}{1+2\epsilon}\right]$, we can establish the following inequality:
\begin{align*}
(1+\epsilon)p^{2}+\epsilon p^{2}\leq p(1-\epsilon)^{2}.
\end{align*}
For the expression 
\begin{align*}
\left(1-\frac{1}{n}\right)\left(1-\frac{\log(1/p)}{2(1-\epsilon)\log n}\right).
\end{align*}
Let $m_{3}=\frac{1}{1-(1-\epsilon)^{1/2}}$. 
For $n>m_{3}$, we have
\begin{align*}
\left(1-\frac{1}{n}\right)\geq(1-\epsilon)^{1/2}.
\end{align*}
Define $m_{4}=\exp\left(\frac{\log(1/p)}{(1-(1-\epsilon)^{1/2})2(1-\epsilon)}\right)$. 
For $n>m_{4}$, it follows that
\begin{align*}
\left(1-\frac{\log(1/p)}{2(1-\epsilon)\log n}\right)\geq(1-\epsilon)^{1/2}.
\end{align*}
By combining these two conditions, we set $n_{0}^{\prime\prime}=\max\left\{m_{3},m_{4}\right\}$. 
For $n>n_{0}^{\prime\prime}$, then
\begin{align*}
(1-\epsilon)\leq \left(1-\frac{1}{n}\right)\left(1-\frac{\log(1/p)}{2(1-\epsilon)\log n}\right).
\end{align*}
Thus, we get
\begin{align*}
(1+\epsilon)p^{2}+\epsilon p^{2}\leq p(1-\epsilon)^{2}\leq p(1-\epsilon)\left(1-\frac{1}{n}\right)\left(1-\frac{\log(1/p)}{2(1-\epsilon)\log n}\right).
\end{align*}
Let $n_{0}=\max\{n_{0}^{\prime},n_{0}^{\prime\prime}\}$ and $n>n_{0}$. 
Then
\begin{align*}
&(1+\epsilon)p^{2}+4p(1-p)n^{-1}+4C_{0}(p(1-p))^{1/2}n^{-7/6}\log n+C_{0}^{2}n^{-4/3}(\log n)^{2}\\
&\leq p(1-\epsilon)\left(1-\frac{1}{n}\right)\left(1-\frac{\log(1/p)}{2(1-\epsilon)\log n}\right).
\end{align*}
The conclusion is proven by multiplying both sides of the above inequality by $n^{2}$. 
\end{proof}

\begin{proof}[Proof of Theorem \ref{thm_1}]
We use the notation $\mathcal{X}$ to represent the inequality
\begin{align*}
\lambda_{1}^{2}+\lambda_{2}^{2}\leq(1+\epsilon)p^{2}n^{2}+4p(1-p)n+4C_{0}(p(1-p))^{1/2}n^{5/6}\log n+C_{0}^{2}n^{2/3}(\log n)^{2}. 
\end{align*}
Similarly, let $\mathcal{Y}$ represent the inequality
\begin{align*}
2p(1-\epsilon)\frac{n(n-1)}{2}\left(1-\frac{\log(1/p)}{2(1-\epsilon)\log n}\right)\leq2p(1-\epsilon)\frac{n(n-1)}{2}\left(1-\frac{1}{\omega(G_{n})}\right),
\end{align*}
and let $\mathcal{Z}$ denote the inequality 
\begin{align*}
2p(1-\epsilon)\frac{n(n-1)}{2}\left(1-\frac{1}{\omega(G_{n})}\right)\leq2e(G_{n})\left(1-\frac{1}{\omega(G_{n})}\right).
\end{align*}
From now on, we fix $\epsilon$ and $n_{0}$ given by Lemma \ref{lem_1} to obtain that for all $n> n_{0}$, it follows that
\begin{align*}
\mathbb{P}\left(\lambda_{1}^{2}+\lambda_{2}^{2}\leq 2e(G_{n})\left(1-\frac{1}{\omega(G_{n})}\right)\right)
&\geq\mathbb{P}\left(\textup{$\mathcal{X}$ and $\mathcal{Y}$ and $\mathcal{Z}$}\right)\\
&\geq\mathbb{P}\left(\mathcal{X}\right)+\mathbb{P}\left(\mathcal{Y}\right)+\mathbb{P}\left(\mathcal{Z}\right)-2.
\end{align*}
Let $0<\epsilon^{\prime}\leq(1+\epsilon)^{1/2}-1$. 
By Lemma \ref{lem_Juhasz}, there exists $n_{1}$ such that for $n> n_{1}$, 
\begin{align*}
\mathbb{P}\left(\lambda_{1}\leq(1+\epsilon^{\prime})pn\right)=1,
\end{align*}
which implies
\begin{align*}
\mathbb{P}\left(\lambda_{1}^{2}\leq(1+\epsilon^{\prime})^{2}p^{2}n^{2}\right)=1. 
\end{align*}
Since $\epsilon^{\prime}\leq(1+\epsilon)^{1/2}-1$, we have
\begin{align}\label{equ_1}
\mathbb{P}\left(\lambda_{1}^{2}\leq(1+\epsilon)p^{2}n^{2}\right)=1. 
\end{align}
By Lemma \ref{lem_Furefi}, there exists $n_{2}$ such that for $n> n_{2}$, 
\begin{align*}
\mathbb{P}\left(\lambda_{2}\leq 2(p(1-p)n)^{1/2}+C_{0} n^{1/3}\log n\right)=1,
\end{align*}
which implies
\begin{align}\label{eau_2}
\mathbb{P}\left(\lambda_{2}^{2}\leq 4p(1-p)n+4C_{0}(p(1-p))^{1/2}n^{5/6}\log n+C_{0}^{2}n^{2/3}(\log n)^{2}\right)=1.
\end{align}
Let $n> \max\{n_{1},n_{2}\}$. 
Combining (\ref{equ_1}) and (\ref{eau_2}) yields 
\begin{align*}
\mathbb{P}\left(\mathcal{X}\right)=1.
\end{align*}
For fixed $\epsilon$, by Lemma \ref{lem_grimmett}, we can find $n_{3}$ such that for $n> n_{3}$, 
\begin{align*}
\mathbb{P}\left(\omega(G_{n})\geq(1-\epsilon)\frac{2\log n}{\log(1/p)}\right)=1.
\end{align*}
That implies
\begin{align*}
\mathbb{P}\left(\left(1-\frac{\log(1/p)}{2(1-\epsilon)\log n}\right)\leq\left(1-\frac{1}{\omega(G_{n})}\right)\right)=1.
\end{align*}
Thus, we have
\begin{align*}
\mathbb{P}\left(\mathcal{Y}\right)=1.
\end{align*}
Let $n>n^{\prime}=\max\{n_{0},n_{1},n_{2},n_{3}\}$. 
Then
\begin{align*}
\mathbb{P}\left(\lambda_{1}^{2}+\lambda_{2}^{2}\leq 2e(G_{n})\left(1-\frac{1}{\omega(G_{n})}\right)\right)\geq1+1+\mathbb{P}\left(\mathcal{Z}\right)-2\geq\mathbb{P}\left(\mathcal{Z}\right).
\end{align*}
It is a fact that the excepted number of edges in $G_{n}$ is $p\frac{n(n-1)}{2}$. 
Thus, Hoeffding's inequality implies that
\begin{align*}
\mathbb{P}\left(e(G_{n})\leq(1-\epsilon)p\frac{n(n-1)}{2}\right)\leq \exp\left(-2\frac{\epsilon^{2}\left(\frac{pn(n-1)}{2}\right)^{2}}{\frac{n(n-1)}{2}}\right)
=\exp\left(-\epsilon^{2}p^{2}n(n-1)\right)
\end{align*}
for all $\epsilon\in(0,1)$. 
Equivalently, we obain
\begin{align*}
\mathbb{P}\left(e(G_{n})\geq(1-\epsilon)p\frac{n(n-1)}{2}\right)\geq 1-\exp\left(-\epsilon^{2}p^{2}n(n-1)\right).
\end{align*}
That implies
\begin{align*}
\mathbb{P}\left(\mathcal{Z}\right)\geq 1-\exp\left(-\epsilon^{2}p^{2}n(n-1)\right).
\end{align*}
Thus, we conclude that
\begin{align*}
\mathbb{P}\left(\lambda_{1}^{2}+\lambda_{2}^{2}\leq 2e(G_{n})\left(1-\frac{1}{\omega(G_{n})}\right)\right)\geq 1-\exp\left(-\epsilon^{2}p^{2}n(n-1)\right).
\end{align*}
\end{proof}

\section*{Acknowledgement}

This work is supported by the National Natural Science Foundation of China (No. 12071097, 12371344), the Natural Science Foundation for The Excellent Youth Scholars of the Heilongjiang Province (No. YQ2022A002) and the Fundamental Research Funds for the Central Universities.

\end{CJK*}
\end{spacing}
\end{document}